\documentclass[twoside]{article}
\usepackage[utf8]{inputenc}
\usepackage{amsmath, amsfonts, amsthm}
\usepackage{geometry, graphicx, tikz, scalerel}
\usepackage{algorithm, algpseudocode}
\usepackage{authblk, hyperref}

\newtheorem{theorem}{Theorem}[section]
\newtheorem{corollary}[theorem]{Corollary}
\newtheorem{proposition}[theorem]{Proposition}
\newtheorem{lemma}[theorem]{Lemma}
\theoremstyle{definition}
\newtheorem{definition}[theorem]{Definition}
\newtheorem{example}[theorem]{Example}

\renewcommand{\phi}{\varphi}
\newcommand*{\rdown}[2]{R_{#1}^{\underline{#2}}}
\DeclareMathOperator{\toep}{toep}

\newcommand*{\rsigma}{\sigma_{\hspace{-0.15em}\scaleobj{0.7}{\rightarrow}}}
\newcommand*{\dsigma}{\sigma_{\hspace{-0.15em}\scaleobj{0.7}{\downarrow}}}

\usepackage{datetime}
\newdateformat{daymonthyear}{%
\THEDAY\ \monthname[\THEMONTH] \THEYEAR}

\usepackage{fancyhdr}
\begin{document}

\title{A combinatorial interpretation of Gaussian blur}
\author{Travis Dillon}
\date{\daymonthyear\today}

\maketitle

\begin{abstract}\noindent
Gaussian blur is a commonly-used method to filter image data. This paper introduces the collapsing sum, a new operator on matrices that provides a combinatorial interpretation of Gaussian blur. We study the combinatorial properties of this operator and prove the explicit relation between Gaussian blur and the collapsing sum.
\end{abstract}


\section{Introduction}\label{sec:introduction}

Image data, and data in general, is often filtered to remove \textit{noise}, random fluctuations that hide the underlying pattern. For images, one of the most common solutions is to apply Gaussian blur, which smooths the data to remove noise.

Because of its use, there has been much interest in discovering efficient algorithms for Gaussian blur \cite{recursive-cosine-gblur,elboher-efficient-gblur,gaussian-blur}. Waltz and Miller \cite{gaussian-blur} in particular provide a clear example of the ways in which properties of binomial coefficients can be leveraged to create such an algorithm. An analysis of their algorithm in Section \ref{sec:background} leads to the following definitions.

\begin{definition}\label{dfn:collapsing sum} Let $A$ be a real $m\times n$ matrix. If $m\geq 2$, then the $(m-1)\times n$ matrix $\dsigma (A)$ has entries
\[ \dsigma (A)_{i,j} = a_{i,j} + a_{i+1,j}.\]
If $n\geq 2$, then the $m\times (n-1)$ matrix $\rsigma(A)$ has entries
\[ \rsigma(A)_{i,j} = a_{i,j} + a_{i,j+1}. \]
Finally, the matrix $\sigma(A):=\dsigma\circ\rsigma(A)=\rsigma\circ\dsigma(A)$ is the \textit{collapsing sum} of $A$.
\end{definition}

The collapsing sum captures mathematically what Waltz and Miller describe computationally in \cite{gaussian-blur}. In this paper, we establish the connection between the collapsing sum and Gaussian blur and provide a theoretical study of the combinatorial properties of this operator.

Section 3 provides the main combinatorial analysis of the operator. We recast the collapsing sum (and therefore Gaussian blur) in terms of matrix multiplication and define a new class of matrices called coefficient matrices that generalize Gaussian blur. The section culminates in Theorem \ref{thm:gaussian blur}, which explicitly describes the connection between the collapsing sum and Gaussian blur.

In the remainder of the paper, we turn to the purely combinatorial properties of this operator; for example, we completely describe the fully-collapsed state of Toeplitz matrices (see Proposition \ref{thm:circulant-fully-collapsed}).
Finally, we discuss generalizations of Gaussian blur in connection to Waltz and Miller's algorithm.


\section{Background}\label{sec:background}
As the collapsing sum will be motivated by Gaussian blur, we begin with a description of image filtering. Grayscale images are stored as matrices: Shades of gray are represented as numbers in a particular range (for example, integers from 0 to 255, or real numbers from 0 to 1), and each entry represents a pixel.\footnote{Whether 0 represents black or white depends on the application; in printing, 0 represents white, whereas in computing, 0 represents black. We won't need to pick between these conventions for our purposes.} We will consider only grayscale images, but this is not an artificial restriction; the same techniques are used to apply a filter to color images. The data for color images are stored as three separate values of red, green, and blue. Applying a filter to a color image consists of separating the data into three matrices by color type, applying the filter to each, and recombining.

It may be the case that the image contains noise, so that the pixel values are randomly perturbed by environmental factors. Because noise is random, it seems possible to eliminate it by averaging pixel values in the neighborhood of a central pixel. This process is known as filtering.

Filters are applied in a process called convolution. The matrix that represents the filter is called a \textit{kernel matrix}. Typically, kernel matrices are square with dimensions $(2r+1)\times (2r+1)$. The integer $r$ is the \textit{radius} of the filter and controls the size of the neighborhood. For simplicity in the convolution formula, kernel matrices are indexed so that the central entry has coordinates (0,0). Convolving the kernel matrix $K=(k_{i,j})$ with an $m\times n$ image matrix $A$ returns the  $m\times n$ matrix $K\ast A$ with entries 
\[(K*A)_{p,q} := \sum_{i,j=-r}^r k_{i,j}\cdot a_{p-i,q-j}.\]
The convolution can be equivalently expressed as
\[ (K\ast A)_{p,q} = \sum_{\substack{i+k = p \\ j+\ell=q}} k_{i,j}\cdot a_{k,\ell}. \]
We require that $\sum_{i,j}k_{i,j} = 1$ so that the overall intensity of the image does not change.

As written, however, convolution is not well-defined when $a_{p,q}$ is near a boundary of $A$. In these cases, the convolution formula requires values of entries that don't exist, such as $a_{-1,0}$. To fix this problem, we use what are called \textit{edge-handling techniques}. In this paper, we only consider two common techniques: extending $A$ to have values beyond its edges or applying the filter to only those pixels for which convolution is defined (the latter is called \textit{cropping}).\footnote{See  \url{https://en.wikipedia.org/wiki/Kernel\_(image\_processing)} for a list of edge-handling techniques.} To apply a kernel matrix of radius $r$ to all pixels in an $m\times n$ matrix $A$, we need to extend $A$ by $r$ rows and columns on each side, to a matrix of size $(m+2r) \times (n+2r)$, where the central $m\times n$ block is the matrix $A$. The filter is applied to each pixel in the central $m\times n$ block of the enlarged matrix.

If extension is chosen as the edge-handling technique, let $A'$ denote the corresponding extension of $A$. If cropping is chosen as the edge-handling technique, then set $A' = A$. Applying the filter to $A$ with the chosen edge-handling technique is equivalent to applying the filter to $A'$ with cropping.

The simplest blur filter is the \textit{box blur}. Let $J_{m\times n}$ represent the $m\times n$ matrix with each entry equal to $1$, and abbreviate $J_{n\times n}$ by $J_n$.

\begin{definition}
The kernel matrix $B_{2r+1}$ for the box blur of radius $r$ is $(2r+1)^{-2} J_{2r+1}$.
\end{definition}
As a visual example, consider the following image.

\begin{center}
\includegraphics[scale=0.5]{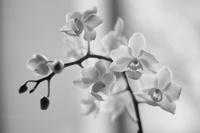}
\end{center}
\noindent
The results of applying box blurs with radii of 1, 2, and 3, respectively, to this image are shown below.

\begin{center}
\includegraphics[scale=0.5]{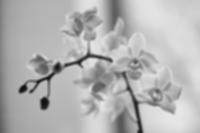}\qquad
\includegraphics[scale=0.5]{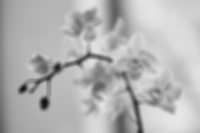}\qquad
\includegraphics[scale=0.5]{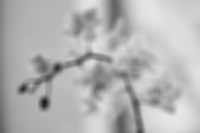}
\end{center}

One problem with box blurs, especially ones of large radius, is that pixels are weighted the same regardless of their distance from the central pixel. It makes sense to weight closer pixels more heavily than distant pixels: Pixels that are closer to each other will contain more information about each other than those that are farther away. Because of this, the \textit{Gaussian blur}, which takes this into account, is more commonly used. The values for the Gaussian blur kernel matrix are derived from the two-dimensional Gaussian curve
\[f(x,y)=\frac{1}{2\pi s^2}e^{-\frac{x^2+y^2}{2s^2}},\]
where $s$ represents the standard deviation of the distribution. Sometimes values are directly sampled from this function, but they are often approximated using binomial coefficients.

\begin{definition}
The $(2r+1)\times (2r+1)$ kernel matrix $G_{2r+1}$ of the approximate Gaussian blur with radius $r$ has entries, for each $-r \leq i,j \leq r$, of
\begin{equation}\label{eq:gaussian-blur-def}
(G_{2r+1})_{i,j} = \frac{1}{4^{2r}}\binom{2r}{i+r}\binom{2r}{j+r}.
\end{equation}
\end{definition}

\begin{example}
The kernel matrix for the $5\times 5$ approximate Gaussian blur is 
\[G_5 = \frac{1}{256}\begin{pmatrix}
1&4&6&4&1\\
4&16&24&16&4\\
6&24&36&24&6\\
4&16&24&16&4\\
1&4&6&4&1
\end{pmatrix}.\]
\end{example}

Notice that the pixels near the center are weighted highest, and that the values taper off toward the edges. Applying Gaussian blurs of radii 1, 2, and 3, respectively, to our example image from above results in the images below. The images appear smooth, while each individual element of the image remains clear.

\begin{center}
\includegraphics[scale=0.5]{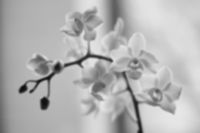}\qquad
\includegraphics[scale=0.5]{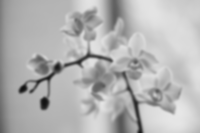}\qquad
\includegraphics[scale=0.5]{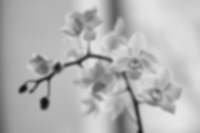}
\end{center}

Each Gaussian blur kernel matrix can be decomposed into the product of a row vector and a column vector. Since it is much faster to compute smaller convolutions than large ones, Gaussian blur algorithms break the computation into two smaller convolutions: one with the row vector, and one with the column vector.

In \cite{gaussian-blur}, Waltz and Miller develop an algorithm for computing Gaussian blur that is more efficient than simple decomposition. The key observation that the authors use is that Gaussian blurs of larger radius can be created through repeated convolution with Gaussian blurs of smaller radius. Their algorithm decomposes the Gaussian blur kernel matrix into a row vector and a column vector, and it decomposes each of these vectors into the repeated convolution of the matrices $\begin{pmatrix} 1 & 1 \end{pmatrix}$ and $\begin{pmatrix} 1 & 1 \end{pmatrix}^T$, respectively. With a bit of clever programming, Waltz and Miller created an algorithm that runs much faster than one that only uses the decomposition property.

Convolution by the matrices $\begin{pmatrix} 1 & 1 \end{pmatrix}$ and $\begin{pmatrix} 1 & 1 \end{pmatrix}^T$ corresponds to the operations $\rsigma$ and $\dsigma$, respectively. This observation leads to Definition \ref{dfn:collapsing sum}.

\begin{example}
Take the $2\times 2$ matrix
\[A=\begin{pmatrix} a_{1,1} & a_{1,2} \\ a_{2,1} & a_{2,2}\end{pmatrix}.\]
Applying the collapsing operations, we get 
\begin{gather*}
\dsigma(A) = \begin{pmatrix} a_{1,1} + a_{2,1} & a_{1,2} + a_{2,2}\end{pmatrix}
\qquad
\rsigma(A) = \begin{pmatrix} a_{1,2} + a_{2,2} \\ a_{2,1} + a_{2,2}\end{pmatrix}
\\ 
\sigma (A)= \begin{pmatrix} a_{1,1} + a_{1,2} + a_{2,1} + a_{2,2} \end{pmatrix}.
\end{gather*}
\end{example}


\section{Equivalence of Gaussian blur and collapsing sum}\label{sec:eq blur and sum}

In this section, we place the collapsing sum on a matrix-theoretic foundation and explicitly connect it with Gaussian blur via Theorem \ref{thm:gaussian blur}.

The following properties of the collapsing sum follow directly from Definition \ref{dfn:collapsing sum}.

\begin{proposition}\label{thm:sigma linearity}
Let $A$ and $B$ be $m\times n$ matrices and $c$ be any real number. Then
\begin{enumerate}
\item $\sigma (A+B)=\sigma (A) + \sigma (B)$,\vspace{-0.3em}
\item $\sigma(cA) = c\cdot\sigma (A)$,\vspace{-0.3em}
\item $\dsigma(A^T) = \rsigma(A)^T$, and\vspace{-0.3em}
\item $\sigma (A^T) = \sigma (A)^T$\vspace{-0.3em}
\end{enumerate}
whenever the operations are defined. The first two statements also hold for $\dsigma$ and $\rsigma$.
\end{proposition}

Much of the investigation will examine repeated application of the collapsing sum. Let $A$ be an $m\times n$ matrix. Then $\sigma^0 (A) = A$, and for each positive integer $1\leq s<\min\{m,n\}$, we define $\sigma^s (A)=\sigma(\sigma^{s-1}(A))$. The operators $\dsigma^s$ and $\rsigma^s$ are defined similarly.

Let $I_m$ be the $m\times m$ identity matrix and $\delta_{i,j}$ be the Kronecker delta function
\[\delta_{i,j} = \begin{cases}
1 & \text{if } i = j \\
0 & \text{if } i \not= j.
\end{cases}\]

\begin{definition}
We denote by $R_m$ the $(m-1) \times m$ matrix with entries $r_{i,j}=\delta_{i,j}+\delta_{i+1,j}$. For a positive integer $k < m$, we define $\rdown{m}{k}$ as the product $R_{m-k+1}R_{m-k+2}\cdots R_m$. Further, let $\rdown{m}{0}=I_{m}$.
\end{definition}

The matrices $R_m$ have $1$'s on the diagonal and superdiagonal and $0$'s elsewhere. The notation $\rdown{m}{k}$ is defined analogously to the falling power notation $n^{\underline{k}} = n(n-1)\cdots (n-k+1)$.

\begin{example}
We have
\[
    R_4 = \begin{pmatrix}
    1 & 1 & 0 & 0 \\
    0 & 1 & 1 & 0 \\
    0 & 0 & 1 & 1
    \end{pmatrix}
\]
and
\[
    \rdown{4}{2} = 
    \begin{pmatrix}
    1 & 1 & 0 \\
    0 & 1 & 1
    \end{pmatrix}
    \begin{pmatrix}
    1 & 1 & 0 & 0 \\
    0 & 1 & 1 & 0 \\
    0 & 0 & 1 & 1
    \end{pmatrix}
    =
    \begin{pmatrix}
    1 & 2 & 1 & 0 \\
    0 & 1 & 2 & 1
    \end{pmatrix}.
\]
\end{example}

\begin{proposition}\label{matrix version single sum}
Let $A$ be an $m \times n$ matrix with $m,n\geq 2$. Then $\dsigma^s(A)=\rdown{m}{s} A$ and $\rsigma^s(A) = A (\rdown{n}{s})^T$.
\end{proposition}
\begin{proof}
First note that $R_m A$ is an $(m-1) \times n$ matrix. Using the definition of $R_m$, the entry $(R_m A)_{i,j}$ is
\[
\sum_{k=1}^{m} (\delta_{i,k} + \delta_{i+1,k})a_{k,j}
=a_{i,j}+a_{i+1,j}
=\dsigma(A)_{i,j}.
\]
A quick induction argument shows that $\dsigma^s(A)=\rdown{m}{s} A$. The calculation for the second assertion is similar.
\end{proof}

Consequently, $\sigma^s(A)=(\rdown{m}{s})A(\rdown{n}{s})^T$.

\begin{proposition}\label{product of sum matrices}
Let $m$ be a positive integer and $s\leq m$ be a nonnegative integer. Then $\rdown{m}{s}$ is an $(m-s)\times m$ matrix with entries $(\rdown{m}{s})_{i,j}=\binom{s}{j-i}$.
\end{proposition}
\begin{proof}
We proceed by induction. For $s=0$, the theorem simplifies to the definition of $I_{m} = \rdown{m}{0}$. Now suppose that the theorem holds for some nonnegative integer $k$. Then $\rdown{m}{k+1}=R_{m-k}\rdown{m}{k}$. Since $\rdown{m}{k}$ is an $(m-k) \times m$ matrix, $\rdown{m}{k+1}$ is an $(m-k-1)\times m$ matrix. Further, by writing $(R_{m-k})_{i,r} = \delta_{i,r} + \delta_{i+1,r}$, we have 
\begin{align*}
(\rdown{m}{k+1})_{i,j}&=\sum_{r=1}^{m-k}(R_{m-k})_{i,r} (\rdown{m}{k})_{r,j}\\
&=\binom{k}{j-i}+\binom{k}{j-(i+1)}\\
&=\binom{k+1}{j-i},
\end{align*}
so the formula holds by induction.
\end{proof}

We now introduce an object that will facilitate the proof of Theorem \ref{thm:gaussian blur}.

\begin{definition}\label{def:coefficient matrix}
Let $a<m$ and $b < n$ be nonnegative integers. The \textit{coefficient matrix} $C_{m\times n}^{a, b} = (c_{i,j})$ is the unique $m\times n$ matrix such that $\sum_{i,j}\dsigma^a\rsigma^b(A)_{i,j}=\sum_{i,j} c_{i,j}a_{i,j}$ for all $m\times n$ matrices $A$. We abbreviate $C^{a,b}_n := C^{a,b}_{n\times n}$ and $C^a_{m\times n} := C^{a,a}_{m\times n}$.
\end{definition}

One interpretation of the coefficient matrix uses indeterminates. Let $X = (x_{i,j})$ be an $m\times n$ matrix of indeterminates; that is, the entries of $X$ are distinct symbols, not numbers. The entry $c_{i,j}$ of the coefficient matrix $C^{a,b}_{m\times n}$ is the sum of the coefficients of $x_{i,j}$ across all entries of $\dsigma^a\rsigma^b(X)$. Thus, one way to think of the coefficient matrix is that its $(i,j)$th entry represents the number of times that $x_{i,j}$ appears in $\dsigma^a\rsigma^b(X)$.

We now work to describe the entries of the coefficient matrices explicitly.

\begin{definition}
Let $A$ be an $m\times n$ matrix and $\mathbf{e}_n$ be the $n\times 1$ vector in which each entry is $1$. The \textit{column sum vector} of $A$ is $\alpha=A^T\mathbf{e}_m$, and the \textit{row sum vector} of $A$ is $\beta = A \mathbf{e}_n$. That is, $\alpha_j$ is the sum of the elements in the $j$th column of $A$, and $\beta_j$ is the sum of the elements of the $j$th row of $A$.
\end{definition}

\begin{lemma}\label{matrix product sum vectors}
Let $X=(x_{i,j})$ be a matrix of indeterminates and $A$ and $B$ be matrices such that the product $AXB$ is defined. If $\alpha$ is the column sum vector of $A$ and $\beta$ is the row sum vector of $B$, then the coefficient of $x_{p,q}$ in the formal expression $\sum_{i,j}(AXB)_{i,j}$ is $\alpha_p\beta_q$.
\end{lemma}
\begin{proof}
Choose any indeterminate $x_{p,q}$. We have
\[\sum_{i,j}(AXB)_{i,j}=\sum_{i=1}^{m}\sum_{j=1}^n\left[\sum_{k=1}^m\sum_{r=1}^n a_{i,k}\cdot x_{k,r}\cdot b_{r,j}\right].\]
We obtain the coefficient of $x_{p,q}$ by summing only those terms where $k=p$ and $r=q$. This coefficient is thus
\[\sum_{i=1}^{m}\sum_{j=1}^n a_{i,p}b_{q,j}=\Bigg[\sum_{i=1}^m a_{i,p}\Bigg]\Bigg[\sum_{j=1}^n b_{q,j}\Bigg].\]
The left term in this product is $\alpha_p$, and the right term is $\beta_q$.
\end{proof}

\begin{proposition}\label{coefficient matrix sum vectors}
Let $\alpha$ be the column sum vector of $\rdown{m}{a}$ and $\beta$ be the column sum vector of $\rdown{n}{b}$. Then $C_{m\times n}^{a,b} = \alpha\beta^T$.
\end{proposition}
\begin{proof}
Let $X$ be an $m\times n$ matrix of indeterminates. Apply Lemma \ref{matrix product sum vectors} to $\dsigma^a\rsigma^b(X)=(\rdown{m}{a})X(\rdown{n}{b})^T$. The row sum vector of $(\rdown{n}{b})^T$ is simply the column sum vector of $\rdown{n}{b}$. The sum in Lemma \ref{matrix product sum vectors} is the sum that defines the coefficient matrix, so the $(i,j)$th entry of the coefficient matrix $C^{a,b}_{m\times n}$ is $(\alpha\beta^T)_{i,j} = \alpha_i\beta_j$.
\end{proof}

Proposition \ref{coefficient matrix sum vectors} implicitly gives the following formula for coefficient matrices.

\begin{corollary}\label{intermediate collapse}
Let $m$ and $n$ be positive integers and $a < m$ and $b < n$ be nonnegative integers. The coefficient matrix $C_{m\times n}^{a,b}$ has entries $c_{i,j} = \big[\!\sum_{\ell=1}^{m-a}\binom{a}{i-\ell}\big] \! \big[\!\sum_{\ell=1}^{n-b}\binom{b}{j-\ell}\big]$.
\end{corollary}
\begin{proof}
Proposition \ref{product of sum matrices} shows that
\[ \alpha_{i} = \sum_{\ell=1}^{m-s} (\rdown{m}{a})_{\ell,i} = \sum_{\ell=1}^{m-a}\binom{a}{i-\ell}. \]
A similar calculation holds for $\beta_j$.
\end{proof}

\begin{corollary}\label{square collapsed}
Let $A$ be an $m\times n$ matrix. The value of the single entry of the matrix $\dsigma^{m-1}\rsigma^{n-1} (A)$ is $\sum_{i=1}^m\sum_{j=1}^n \binom{m-1}{i-1}\binom{n-1}{j-1} a_{i,j}$.
\end{corollary}
\begin{proof}
Let $C_{m\times n}^{m-1,n-1} = (c_{i,j})$. Since $\dsigma^{m-1}\rsigma^{n-1} (A)$ has a single entry, we have
\[ \sigma^{n-1}(A)_{1,1} = \sum_{i,j}\sigma^{n-1}(A)_{i,j} = \sum_{i,j} c_{i,j}a_{i,j} \]
by the definition of the coefficient matrix. Corollary \ref{intermediate collapse} shows that $c_{i,j} = \binom{m-1}{i-1}\binom{n-1}{j-1}$.
\end{proof}

The entries of $\sigma^s(A)$ are determined by the blocks of $A$ of size $(s+1)\times(s+1)$. From this observation, Corollary \ref{square collapsed} can be used to find the value of any entry in $\sigma^s(A)$: simply apply the corollary to the submatrix $(a_{p+i,q+j})_{i,j=0}^s$ to determine the value of $\sigma^s(A)_{i,j}$.

Recall that to apply a kernel matrix, we need to specify an edge-handling technique, wherein we extend the matrix $A$ to a matrix $A'$. Then applying the filter to $A$ with the edge-handling technique is equivalent (by definition) to applying the filter to $A'$ with cropping.

\begin{theorem}\label{thm:gaussian blur}
Suppose a matrix $A$ and an edge-handling technique yielding the extension $A'$ of $A$ are given. Then $G_{2r+1} \ast A = 4^{-2r}\sigma^{2r}(A')$ for all nonnegative integers $r$.
\end{theorem}
\begin{proof}
Each entry of $G_{2r+1} \ast A$ corresponds to a block of $A'$ of size $(2r+1)\times (2r+1)$. From equation (\ref{eq:gaussian-blur-def}), the value of the entry $(G_{2r+1}\ast A)_{p,q}$ is
\[ \frac{1}{4^{2r}} \sum_{i=-r}^r\,\sum_{j=-r}^r\binom{2r}{i+r}\binom{2r}{j+r}a'_{p+i,q+i}. \]

On the other hand, let $B := (a'_{p+i,q+j})_{i,j=-r}^r$ be a submatrix of $A'$. Applying Corollary \ref{square collapsed} and then (\ref{eq:gaussian-blur-def}) gives
\begin{align*}
\sigma^{2r}(A')_{p,q} &= \sigma^{2r}(B)_{1,1}\\
&= \sum_{i=0}^{2r}\,\sum_{j=0}^{2r} \binom{2r}{i}\binom{2r}{j} b_{i,j}\\
&= \sum_{i=-r}^r\,\sum_{j=-r}^r\binom{2r}{i+r}\binom{2r}{j+r}a'_{p+i,q+i}\\
&= 4^{2r} (G_{2r+1}\ast A)_{p,q}. \qedhere
\end{align*}
\end{proof}

Theorem \ref{thm:gaussian blur} may be equivalently stated as an equality of operators:
\[ G_{2r+1} = 4^{-2r}\sigma^{2r}. \]\vspace{-2.2em}


\section{Further properties of the collapsing sum}\label{sec:further properties}

\subsection{Special classes of matrices}

\begin{definition}
A \textit{Toeplitz matrix} is an $m\times n$ matrix $A$ with the property that $a_{i,j} = a_{k,\ell}$ whenever $i-j = k-\ell$. We denote by $\toep(a_{-n+1},\dots,a_{m-1})$ the $m\times n$ Toeplitz matrix $A$ with entries $a_{i,j} = a_{i-j}$.
\end{definition}

\begin{example}
The general $4 \times 5$ Toeplitz matrix $\toep(a_{-4},\dots,a_{3})$ is
\[ \begin{pmatrix}
a_0 & a_{-1} & a_{-2} & a_{-3} & a_{-4}\\
a_{1} & a_0 & a_{-1} & a_{-2} & a_{-3}\\
a_{2} & a_{1} & a_0 & a_{-1} & a_{-2}\\
a_{3} & a_{2} & a_{1} & a_0 & a_{-1}
\end{pmatrix}. \]
\end{example}
Toeplitz matrices have applications in a wide variety of pure and applied areas, including representation theory, signal processing, differential and integral equations, and quantum mechanics. Moreover, every $n\times n$ matrix can be decomposed as the product of at most $2n+5$ Toeplitz matrices \cite{toeplitz-decomp}.

In what follows, we use $(m+1)\times (n+1)$ Toeplitz matrices $\toep(a_{-n},\dots, a_m)$ to slightly simplify the statements of the results.

\begin{proposition}\label{thm:circulant-fully-collapsed}
Let $A = \toep(a_{-n},\dots,a_{m})$ be an $(m+1)\times (n+1)$ Toeplitz matrix. Then $\sum_{k=-n}^{m} \binom{m+n}{n+k}a_k$ is the single entry of $\dsigma^{m}\rsigma^{n}(A)$.
\end{proposition}
\begin{proof}
We can decompose the Toeplitz matrix into ``stripes'':
\[ A = \sum_{k=-n}^{m} \toep(0,\dots,0,a_k,0,\dots,0). \]
Since by Proposition \ref{thm:sigma linearity} the collapsing sum distributes over addition, we need only consider the case when one $a_k$ is nonzero. Moreover, since $\sigma(cA) = c\sigma(A)$, we can restrict to $a_k=1$.

Therefore, suppose $ A = \toep(0,\dots,0,1,0,\dots,0)$, so that
\[ a_{i,j} = \begin{cases}
1 & \text{if } i-j = k\\
0 &\text{otherwise.}
\end{cases} \]
We use the convention that $\binom{n}{r}=0$ if $r < 0$ or $r > n$. Applying Corollary \ref{square collapsed} and the binomial symmetry $\binom{n}{r} = \binom{n}{n-r}$ gives
\begin{align*}
\dsigma^{m}\rsigma^{n}(A)_{1,1} &= \sum_{i=0}^{m}\, \sum_{j=0}^{n} \binom{m}{i}\binom{n}{j}a_{i+1,j+1}\\
&= \sum_{i=0}^{m} \binom{m}{i}\binom{n}{i-k}\\
&=\sum_{i=0}^{m} \binom{m}{i}\binom{n}{(n+k)-i}.
\end{align*}
Applying the well-known identity
$ \sum_{i=0}^{m} \binom{m}{i}\binom{n}{r-i} = \binom{m+n}{r} $
finishes the proof.
\end{proof}

A direct application of Proposition \ref{thm:circulant-fully-collapsed} yields the following.

\begin{corollary}
The single entry in $\sigma^n(I_{n+1})$ is the central binomial coefficient $\binom{2n}{n}$.
\end{corollary}

A similar result holds for coefficient matrices.

\begin{proposition}
The single entry in $\sigma^n(C_{n+1}^{n})$ is $\binom{2n}{n}^2$.
\end{proposition}
\begin{proof}
From Corollary \ref{square collapsed}, we have
\[
\sigma^n(C_{n+1}^{n})
= \sum_{i=0}^n \, \sum_{j=0}^n \binom{n}{i}\binom{n}{j}c_{i,j}
= \sum_{i=0}^n \, \sum_{i=0}^n \binom{n}{i}^2\binom{n}{j}^2.
\]
Using the identity $\sum_{k=0}^n \binom{n}{k}^2 = \binom{2n}{n}$ completes the proof.
\end{proof}

Recall that $J_{m\times n}$ denotes the $m\times n$ matrix with each entry equal to $1$.

\begin{proposition}\label{thm:sum-entries-coeff-matrix}
Let $a < m$ and $b < n$ be nonnegative integers. Then $\sum_{i,j}(C_{m\times n}^{a,b})_{i,j}=2^{a+b}(m-a)(n-b)$.
\end{proposition}
\begin{proof}
It follows from Definition \ref{def:coefficient matrix} that the sum of the entries of $C^{a,b}_{m\times n}$ is equal to the sum of the entries of $\dsigma^a\rsigma^b(J_{m\times n})$. Each $(a+1) \times (b+1)$ block of $J_{m\times n}$ is $J_{(a+1)\times (b+1)}$, so $\dsigma^a\rsigma^b(J_{m\times n})_{i,j} = \dsigma^a\rsigma^b(J_{(a+1)\times(b+1)})_{1,1}$ for all $1 \leq i \leq m-a$ and $1 \leq j \leq n-b$. Since $J_{(a+1)\times (b+1)}$ is a Toeplitz matrix, Proposition \ref{thm:circulant-fully-collapsed} gives
\[ \dsigma^a\rsigma^b(J_{(a+1)\times(b+1)})_{1,1} = \sum_{i=0}^{a+b} \binom{a+b}{i} = 2^{a+b}. \]
Noting that $\dsigma^a\rsigma^b(J_{m\times n})$ has $(m-a)(n-b)$ entries completes the proof.
\end{proof}


The proof of Corollary \ref{square collapsed} shows that the Gaussian blur kernel $G_{2r+1}$ is proportional to the coefficient matrix $C_{2r+1}^{2r}$. Similarly, the box blur kernel $B_{2r+1}$ is proportional to the coefficient matrix $C_{2r+1}^0$. The constant of proportionality in both cases is $(\sum_{i,j}(C^s_{2r+1})_{i,j})^{-1}$, where $s=2r$ or $s=0$, respectively. Thus, the coefficient matrices are a generalization that unite these two filters. That is, the expression $(2^{a+b}(m-a)(n-b))^{-1}C^{a,b}_{m\times n}$ provides an interpolation between box blur and Gaussian blur.


\subsection{Further connections with Gaussian blur}

Waltz and Miller \cite{gaussian-blur} extend their techniques to non-square blurs, that is, Gaussian-like blurs using non-square kernel matrices. These can be defined in parallel to the square Gaussian blurs. If $G_{a\times b}$ denotes the kernel for the $a\times b$ Gaussian blur, then
\[(G_{a\times b})_{i,j} = 2^{-(a+b-2)}\binom{a-1}{i-1}\binom{b-1}{j-1}.\]
Since either $a$ or $b$ might be even, there may be no central element, so here we index from $(1,1)$ in the top left corner of the matrix.

The kernel matrix $G_{a\times b}$ is proportional to the coefficient matrix for a fully collapsed $a\times b$ matrix. This extends Theorem \ref{thm:gaussian blur}, since convolving $G_{a\times b}$ with a matrix $A$ is equivalent to applying $2^{-(a+b-2)}\dsigma^{a-1}\rsigma^{b-1}$ to the extended matrix $A'$.

The authors also venture into higher dimensions and discuss higher-dimensional blurs. We can easily transfer this idea to the language of the collapsing sum. Suppose we want to collapse (or, equivalently, blur) an $n$-dimensional array. We can define $\sigma_{\vec{\imath}}$, for $1\leq i\leq n$, to be the operator that ``collapses'' the array in the $i$th direction, akin to the effects of $\dsigma$ and $\rsigma$ in two dimensions. Define $\sigma_n := \sigma_{\vec{1}}\cdots\sigma_{\vec{n}}$. Then powers of $2^{-n}\sigma_n$ give the higher-dimensional blur that Waltz and Miler describe. As before, general rectangular blurs are obtained by simply composing the operators $\frac{1}{2}\sigma_{\vec{\imath}}$ for various values of $i$.


\subsection{A generalized collapsing sum}\label{sec:generalized-sum}

Waltz and Miller's algorithm for Gaussian blur may be extended to an operator that returns weighted sums of entries.

\begin{definition}\label{generalized sigma def}
Let $\gamma$ be an $b_1\times b_2$ matrix and $A$ be an $m\times n$ matrix with $m,n\geq \max\{b_1,b_2\}$. Then $\sigma_{\gamma}(A)$ is an $(m-b_1)\times (n-b_2)$ matrix with
\[ \sigma_{\gamma}(A)_{p,q}= \sum_{i=0}^{b_1-1}\sum_{j=0}^{b_2-1} \gamma_{i+1,j+1}a_{p+i,q+j}. \]
\end{definition}

If $\gamma = \left(\begin{smallmatrix}1&1\\1&1\end{smallmatrix}\right)$, then we recover the original collapsing sum. Moreover, if $\gamma = (1\, 1)$, then $\sigma_\gamma = \rsigma$, and $\sigma_{\gamma^T} = \dsigma$.

For any matrix $\gamma$ of rank $1$, there exist two column vectors $\rho$ and $\phi$ such that $\gamma=\rho\phi^T$. Waltz and Miller's algorithm may be easily adapted for any $2\times 2$ rank-1 matrix. Our previous results on the collapsing sum may also be extended to $\sigma_\gamma$ for any (not necessarily square) matrix $\gamma$ of rank $1$.

\begin{definition}
Let $\phi$ be a column vector with $k$ entries. Let $\phi$ be a column vector with $k$ entries. The $(m-k+1)\times m$ matrix $R^\phi_m$ has entries $(R^\phi_m)_{p,q}=\sum_{i=0}^{k-1}\phi_{i+1}\delta_{p+i,q}$.
\end{definition}

Again, notice that if $\phi=(1\, 1)$, then $R^\phi_m=R_m$. We define the falling powers of these matrices analogously to those of $R_m$. A generalized form of Proposition \ref{coefficient matrix sum vectors} holds in that, for any $m\times n$ matrix $A$,
\[\sigma_{\rho}^a\sigma_{\phi^T}^b(A) = (R_m^\rho)^{\underline{a}} \, A \, [(R_n^{\phi^T})^{\underline{b}}]^T.\]
In particular, if $\gamma=\rho\phi^T$, then
\[ \sigma_\gamma^s(A) = (R_m^\rho)^{\underline{s}} \, A \, [(R_n^{\phi^T})^{\underline{s}}]^T. \]
Similar extensions may be obtained for other results, including the entries of the corresponding coefficient matrices.

\section{Conclusion}

By introducing the collapsing sum operators, we have provided a new combinatorial way to view Gaussian blur. We established the close connection between these concepts and also established a collection of theoretical results on the collapsing sum.

It would be interesting to study the collapsing sum as a matrix operator in its own right. For example, if $G$ is an abelian group and $G^{m\times n}$ represents the additive group of $m\times n$ matrices with entries in $G$, then the collapsing sum is a map from $G^{m\times n}$ to $G^{(m-1)\times (n-1)}$. What are the combinatorial and algebraic properties of this map?

\section*{Acknowledgements}
The author would like to thank Samuel Gutekunst for his invaluable guidance and support, as well as Mike Orrison, Elizabeth Sattler, and the anonymous reviewers, whose insightful comments and suggestions greatly increased the quality of this paper.


\bibliography{collapsing-sum-bibliography}
\bibliographystyle{abbrv}

\end{document}